\documentclass[12pt]{amsart}
\usepackage[margin=1in]{geometry}
\usepackage{amssymb,amsfonts,amsmath}
\usepackage{hyperref}
\usepackage[capitalise]{cleveref}
\newtheorem{theorem}{Theorem}[section]

\newtheorem{lemma}[theorem]{Lemma}
\theoremstyle{remark}
\newtheorem*{remark}{Remark}
\numberwithin{equation}{section}

\title[]{Effective quantum unique ergodicity for Hecke--Maa{\ss} newforms and Landau--Siegel zeros}

\author{Jesse Thorner}
\address{Department of Mathematics, University of Illinois, Urbana, IL 61801, United States}
\email{jesse.thorner@gmail.com}

\begin{document}

\begin{abstract}
We show that Landau--Siegel zeros for Dirichlet $L$-functions do not exist or quantum unique ergodicity for $\mathrm{GL}_2$ Hecke--Maa\ss{} newforms holds with an effective rate of convergence.  This follows from a more general result:  Landau--Siegel zeros of Dirichlet $L$-functions repel the zeros of all other automorphic $L$-functions from the line $\mathrm{Re}(s)=1$.
\end{abstract}

\maketitle

\section{Introduction and statement of the main results}

Let $(\chi_i\pmod{q_i})_{i=1}^{\infty}$ be a sequence of primitive quadratic Dirichlet characters to increasing moduli $q_i$ whose Dirichlet $L$-functions $L(s,\chi_i)$ have a greatest real zero $\beta_i=1-\frac{\lambda_i}{\log q_i}$.  It remains an open problem to rule out the possibility that $\lim_{q_i\to\infty}\lambda_i = 0$, in which case the zeros $\beta_i$ are called Landau--Siegel zeros.  Siegel proved that for all $\varepsilon>0$, there exists an ineffective constant $c(\varepsilon)>0$ such that $\lambda_i\geq c(\varepsilon)q_i^{-\varepsilon}$.  The existence of Landau--Siegel zeros violates the generalized Riemann hypothesis (GRH), induces inequities in the distribution of primes in arithmetic progressions, and obfuscates the study of class numbers of quadratic fields \cite{Davenport}.  However, the existence of Landau--Siegel zeros implies many desirable results, including the infinitude of twin primes \cite{HB83}, the existence of primes in very short intervals \cite{FI04}, the infinitude of very large gaps between primes \cite{KevinFord,Granville_Siegel}, and the existence of a large proportion of nonvanishing central values of Dirichlet $L$-functions to prime moduli \cite{Pratt}.

Let $\Gamma=\mathrm{SL}_2(\mathbb{Z})$, let $\Delta=-y^2(\frac{\partial^2}{\partial x^2}+\frac{\partial^2}{\partial y^2})$ be the hyperbolic Laplacian on $\Gamma\backslash\mathbb{H}$, and let $(\varphi)$ be a basis of Hecke--Maa\ss{} newforms, orthonormal with respect to the Petersson inner product
\[
\langle\Psi,\Xi\rangle = \int_{\Gamma\backslash\mathbb{H}}\Psi(x+iy)\overline{\Xi(x+iy)}d\mu,\qquad \Psi,\Xi\in L^2(\Gamma\backslash\mathbb{H},\mu),
\]
which are eigenfunctions of $\Delta$ (with eigenvalues $\lambda_{\varphi}>\frac{1}{4}$) and all Hecke operators.  Consider the probability measures $d\mu_{\varphi}:=|\varphi(x+iy)|^2 d\mu$, where $d\mu:=y^{-2}dxdy$ is the usual hyperbolic measure; note that $\mu(\Gamma\backslash\mathbb{H})=\frac{\pi}{3}$.  For a test function $\Psi\in L^2(\Gamma\backslash\mathbb{H},\mu)$, we have $\mu_{\varphi}(\Psi) = \langle \Psi,|\varphi|^2 \rangle$ and $\mu(\Psi) = \langle\Psi,1\rangle$.  Lindenstrauss \cite{Lindenstrauss} and Soundararajan \cite{Sound_QUE_Maass} proved the (arithmetic) quantum unique ergodicity conjecture of Rudnick and Sarnak \cite{RS}:
\[
\lim_{\lambda_{\varphi}\to\infty}\mu_{\varphi}(\Psi) = \frac{3}{\pi}\mu(\Psi).
\]
Thus, as $\lambda_{\varphi}\to\infty$, the $L^2$ mass of $\varphi$ equidistributes on $\Gamma\backslash\mathbb{H}$ with respect to $\mu$.  No rate of convergence in terms of $\lambda_{\varphi}$ is known unconditionally, but an optimal rate is closely tied to GRH for a certain family of $L$-functions (see below).

We relate the rate of convergence in quantum unique ergodicity to the existence of Landau--Siegel zeros.  We also prove a new bound on the set of exceptional $\varphi$ for which an effective rate of convergence is not yet unconditionally known.  All implied constants are effectively computable.

\begin{theorem}
\label{thm:QUE}
	Let $\Psi\in L^2(\Gamma\backslash\mathbb{H},\mu)$ be a test function, and let $\varphi$ denote a Hecke--Maa{\ss} newform with Laplace eigenvalue $\lambda_{\varphi}$.
\begin{enumerate}
\item If $\chi\pmod{q}$ is a primitive quadratic Dirichlet character, $A\geq 1$, and $0<\lambda<A^{-1}$, then at least one of the following statements is true:
	\begin{enumerate}
		\item $L(s,\chi)\neq 0$ for $s\geq 1-\lambda/\log q$.
		\item If $\lambda_{\varphi}\in[q,q^{A}]$, then $\mu_{\varphi}(\Psi)=\frac{3}{\pi}\mu(\Psi)+O_{\Psi}((A\lambda)^{\frac{1}{10^{20}}}\log \lambda_{\varphi})$.
	\end{enumerate}
\item Let $M\geq 1$ and $\delta>2$.  We have $\mu_{\varphi}(\Psi) = \frac{3}{\pi}\mu(\Psi)+O_{\Psi}((\log\lambda_{\varphi})^{1-\frac{\delta}{2}})$ for all except $O((\log M)^{10^{20}\delta})$ Hecke--Maa\ss{} newforms $\varphi$ with $\lambda_{\varphi}\leq M$.
\end{enumerate}
\end{theorem}

Soundararajan \cite[p. 1477]{Sound_weak} proved a version of \cref{thm:QUE}(2) with an exceptional set of size $\ll_{\varepsilon}M^{\varepsilon}$ for all $\varepsilon>0$, so \cref{thm:QUE}(2) gives a noticeable improvement.  (Note that the Weyl law tells us that there are $\sim\frac{1}{12}M$ Hecke--Maa{\ss} newforms $\varphi$ with $\lambda_{\varphi}\leq M$.)  \cref{thm:QUE}(1) shows that this exceptional set is empty if there exists a suitable sequence of primitive quadratic Dirichlet characters whose $L$-functions have Landau--Siegel zeros.

The connection between $L$-functions and quantum unique ergodicity is well-understood \cite{IS,Michel,Watson}.  The test function $\Psi$ has an orthonormal expansion with respect to the Petersson inner product in terms of the constant function; a basis of Hecke--Maa\ss{} newforms $(u_{k})_{k=1}^{\infty}$, which are eigenfunctions of $\Delta$ and all Hecke operators; and the Eisenstein series $E(z,\tfrac{1}{2}+it)$ with $t\in\mathbb{R}$ \cite[Chapter 15]{IK}.  Write $\lambda_k=\frac{1}{4}+t_k^2>\frac{1}{4}$ for the Laplace eigenvalue of $u_{k}$.  By orthonormality, we have that $\langle u_k,u_k\rangle=1$.

Let $\varphi$ be a Hecke--Maa\ss{} newform with $\lambda_{\varphi} = \frac{1}{4}+t_{\varphi}^2>\frac{1}{4}$, normalized so that $\langle \varphi,\varphi\rangle=1$.  Since $\mu_{\varphi}(\Psi)=\langle \Psi,|\varphi|^2 \rangle$ and $\mu(\Psi)=\langle \Psi,1\rangle$, the spectral theorem \cite[Theorem 15.5]{IK} implies that
\begin{equation}
\label{eqn:spectral}
\mu_{\varphi}(\Psi)-\frac{3}{\pi}\mu(\Psi)= \sum_{k=1}^{\infty}\langle \Psi,u_{k}\rangle \langle u_k ,|\varphi|^2 \rangle+\frac{1}{4\pi}\int_{\mathbb{R}}\langle\Psi, E(\cdot,\tfrac{1}{2}+it)\rangle \langle E(\cdot,\tfrac{1}{2}+it) ,|\varphi|^2 \rangle dt.
\end{equation}
Since $\Psi$ is a fixed test function, the $\lambda_{\varphi}$-dependence in the rate at which $\mu_{\varphi}(\Psi)$ tends to $\frac{3}{\pi}\mu(\Psi)$ is completely dictated by $\lambda_{\varphi}$-dependence in bounds for the inner products $\langle u_k ,|\varphi|^2 \rangle$ and $\langle E(\cdot,\tfrac{1}{2}+it),|\varphi|^2 \rangle$, where $t\in\mathbb{R}$ and $u_k$ is a fixed Hecke--Maa\ss{} newform.

Watson \cite[Theorem 3]{Watson} proved an elegant relationship between the inner product $\langle u_k,|\varphi|^2 \rangle$ and central values of $L$-functions:
\begin{align*}
|\langle u_k ,|\varphi|^2 \rangle|^2 &= \frac{1}{8}\frac{\Lambda(\frac{1}{2},u_k\otimes \varphi\otimes\widetilde{\varphi})}{\Lambda(1,\mathrm{Ad}^2 u_k)\Lambda(1,\mathrm{Ad}^2\varphi)^2}\\
&\ll \Big|\frac{\Gamma(\frac{\frac{1}{2}+2it_{\varphi}+it_k}{2}) \Gamma(\frac{\frac{1}{2}+2it_{\varphi}-it_k}{2}) \Gamma(\frac{\frac{1}{2}+it_k}{2})^2}{\Gamma(\frac{1+2it_{\varphi}}{2})^2\Gamma(\frac{1+2it_k}{2})}\Big|^2 \frac{L(\frac{1}{2},u_k)L(\frac{1}{2},\mathrm{Ad}^2\varphi\otimes u_k)}{L(1,\mathrm{Ad}^2 u_k)L(1,\mathrm{Ad}^2 \varphi)^2},
\end{align*}
where $\mathrm{Ad}^2$ denotes the adjoint square lift.  (See \cref{sec:L-functions} for the pertinent $L$-function notation.)  An application of Stirling's formula yields
\[
|\langle u_k ,|\varphi|^2 \rangle|\ll_{u_{k}} \lambda_{\varphi}^{-1/4}L(1,\mathrm{Ad}^2\varphi)^{-1}L(\tfrac{1}{2},u_k\otimes \mathrm{Ad}^2\varphi)^{1/2}.
\]
Goldfeld, Hoffstein, and Lieman \cite[Appendix]{HL} proved that\footnote{Their main result is $L(s,\mathrm{Ad}^2\varphi)\neq 0$ for $s\geq 1-c/\log(\lambda_{\varphi}+3)$; one can check that $c=1/500$ is permissible.} $L(1,\mathrm{Ad}^2\varphi )^{-1}\ll \log \lambda_{\varphi}$, hence
\begin{equation}
\label{eqn:inner_discrete}
\begin{aligned}
|\langle u_k ,|\varphi|^2 \rangle| \ll_{u_{k}} \lambda_{\varphi}^{-1/4}(\log \lambda_{\varphi})L(\tfrac{1}{2},u_k\otimes \mathrm{Ad}^2\varphi)^{1/2}.
\end{aligned}
\end{equation}
Using the unfolding method (see \cite[Equation 4.2]{RK}, for example), one computes
\[
|\langle E(\cdot,\tfrac{1}{2}+it) ,|\varphi|^2 \rangle| =\frac{\sqrt{\pi}}{2} \Big|\frac{\Gamma(\frac{\frac{1}{2}+2it_{\varphi}+it}{2}) \Gamma(\frac{\frac{1}{2}+2it_{\varphi}-it}{2}) \Gamma(\frac{\frac{1}{2}+it}{2})^2}{\Gamma(\frac{1+2it_{\varphi}}{2})^2\Gamma(\frac{1+2it}{2})}\Big| \frac{|\zeta(\frac{1}{2}+it) L(\frac{1}{2}+it,\mathrm{Ad}^2\varphi)|}{L(1,\mathrm{Ad}^2\varphi) |\zeta(1+2it)|}.
\]
Stirling's formula and standard bounds for the Riemann zeta function $\zeta(s)$ imply that there exists an absolute and effectively computable constant $A_1>0$ such that
\[
|\langle E(\cdot,\tfrac{1}{2}+it) ,|\varphi|^2 \rangle|\ll (3+|t|)^{A_1}\lambda_{\varphi}^{-1/4}L(1,\mathrm{Ad}^2\varphi)^{-1}|L(\tfrac{1}{2}+it,\mathrm{Ad}^2\varphi )|.
\]
Again, we invoke the work of Goldfeld, Hoffstein, and Lieman to deduce the bound
\begin{equation}
\label{eqn:inner_continuous}
|\langle E(\cdot,\tfrac{1}{2}+it) ,|\varphi|^2 \rangle| \ll (3+|t|)^{A_1} \lambda_{\varphi}^{-1/4}(\log \lambda_{\varphi})|L(\tfrac{1}{2}+it,\mathrm{Ad}^2\varphi )|.
\end{equation}

Let $g$ be a increasing function with $\lim_{x\to\infty}g(x)=\infty$.  Equations \eqref{eqn:spectral}-\eqref{eqn:inner_continuous} show that if there exists an absolute constant $A_2>0$ such that
\begin{equation}
\label{eqn:RS_Watson}
L(\tfrac{1}{2},u_k\otimes \mathrm{Ad}^2\varphi)\ll_{u_k}\frac{\lambda_{\varphi}^{1/2}}{g(\lambda_{\varphi})^2}\qquad\textup{and}\qquad |L(\tfrac{1}{2}+it,\mathrm{Ad}^2\varphi)|\ll (3+|t|)^{A_2}\frac{\lambda_{\varphi}^{1/4}}{g(\lambda_{\varphi})}
\end{equation}
for every $t\in\mathbb{R}$ and every fixed Hecke--Maa\ss{} form $u_k$, then
\begin{equation}
\label{eqn:rate_convgc}
\mu_{\varphi}(\Psi)=\frac{3}{\pi}\mu(\Psi) + O_{\Psi}\Big(\frac{\log\lambda_{\varphi}}{g(\lambda_{\varphi})}\Big).
\end{equation}
Gelbart and Jacquet \cite{GJ} proved that $L(s,\mathrm{Ad}^2\varphi)$ is the $L$-function of a cuspidal automorphic representation of $\mathrm{GL}_3$, and the combined work of Kim and Shahidi \cite{KS2} and Ramakrishnan and Wang \cite{Ram_Wang} shows that $L(s,u_k\otimes \mathrm{Ad}^2\varphi)$ is the $L$-function of a cuspidal automorphic representation of $\mathrm{GL}_6$ when $\varphi$ is not a twist o.  Therefore, the rate of convergence in \eqref{eqn:rate_convgc} is directly tied to bounds for certain $\mathrm{GL}_3$ and $\mathrm{GL}_6$ automorphic $L$-functions on the critical line $\mathrm{Re}(s)=\frac{1}{2}$.

We address this in a broader context.  Let $\mathbb{A}$ be the ring of adeles over $\mathbb{Q}$, and let $\mathfrak{F}_{m}$ be the set of all cuspidal automorphic representations of $\mathrm{GL}_m(\mathbb{A})$ with unitary central character (which we normalize to be trivial on the diagonally embedded copy of the reals so that $\mathfrak{F}_{m}$ is a discrete set).  Let $m\geq 2$ and $\pi\in\mathfrak{F}_{m}$.  The analytic conductor $C(\pi)$ of $\pi$ (see \eqref{eqn:analytic_conductor_def}) is a useful measure of the spectral and arithmetic complexity of $\pi$.

The work of Heath-Brown \cite{HB}; Luo, Rudnick, and Sarnak \cite{LRS}; and M{\"u}ller and Speh \cite{MS} produces the convexity bound $|L(\frac{1}{2},\pi)|\ll_{m}C(\pi)^{1/4}$.  In many applications, this bound barely fails to be of use, and even small improvements would yield deep results.  GRH implies that $|L(\tfrac{1}{2},\pi)|\leq C(\pi)^{o(1)}$, which yields \eqref{eqn:rate_convgc} with an optimal error term of $O_{\Psi}(\lambda_{\varphi}^{-1/4+o(1)})$.  Subconvexity bounds of the form $|L(\frac{1}{2},\pi)|\ll_m C(\pi)^{1/4-\delta_m}$ for some fixed $\delta_m>0$ are known for many classes of $L$-functions, but a nontrivial error term in \eqref{eqn:rate_convgc} remains out of reach.

Under a weak form of the generalized Ramanujan conjecture (GRC), Soundararajan \cite{Sound_weak} proved that $|L(\tfrac{1}{2},\pi)|\ll_{m,\varepsilon} C(\pi)^{1/4}(\log C(\pi))^{-1+\varepsilon}$ for all $\varepsilon>0$.  By a different method, Soundararajan and the author \cite{ST} unconditionally proved
\begin{equation}
\label{eqn:ST_weak}
|L(\tfrac{1}{2},\pi)|\ll_{m} C(\pi)^{\frac{1}{4}}(\log C(\pi))^{-\frac{1}{10^{17}m^3}}
\end{equation}
for all $m\geq 1$.  We will refine the results in \cite{ST} when Landau--Siegel zeros for Dirichlet $L$-functions exist.  In what follows, let $\mathfrak{F}_{m}(Q)=\{\pi\in\mathfrak{F}_{m}\colon C(\pi)\leq Q\}$.

\begin{theorem}
\label{thm:weak}
	Let $t\in\mathbb{R}$, and Let $m\geq 2$ be an integer.
\begin{enumerate}
\item If $\chi\pmod{q}$ is a primitive quadratic Dirichlet character, $A\geq 1$, and $0<\lambda<A^{-1}$, then at least one of the following statements is true:
	\begin{enumerate}
		\item $L(s,\chi)\neq 0$ for $s\geq 1-\lambda/\log q$.
		\item If $A\geq 1$ and $C(\pi)\in[q,q^{A}]$, then
		\[
		|L(\tfrac{1}{2}+it,\pi)|\ll_{m} (3+|t|)^{O(m)} C(\pi)^{\frac{1}{4}}(\log C(\pi))^{-\frac{1}{4\times 10^{16}m^4}}(A\lambda)^{\frac{1}{4\times 10^{16}m^4}}.
		\]
	\end{enumerate}
\item Let $\delta\geq 0$.  For all except $O_m((\log Q)^{10^{16}m^4 \delta})$ of the $\pi\in\mathfrak{F}_{m}(Q)$, we have
\[
|L(\tfrac{1}{2}+it,\pi)|\ll_m (3+|t|)^{O(m)} C(\pi)^{\frac{1}{4}}(\log C(\pi))^{-\delta}.
\]
\end{enumerate}
\end{theorem}
\begin{remark}
Brumley and Mili{\'c}evi{\'c} \cite{BM} proved that $\#\mathfrak{F}_{m}(Q)\gg_m Q^{m+1}$, so the exceptional set in \cref{thm:weak}(2) is quite small.  \cref{thm:weak}(1) shows that this exceptional set is empty if there exists a suitable sequence of primitive quadratic characters with Landau--Siegel zeros.
\end{remark}
\begin{proof}[Proof of \cref{thm:QUE}]
We have $C(u_k\otimes \mathrm{Ad}^2\varphi)\asymp_{u_k}\lambda_{\varphi}^2$ and $C(\mathrm{Ad}^2\varphi)\asymp\lambda_{\varphi}$, so the theorem follows from \cref{thm:weak} and the fact \eqref{eqn:RS_Watson} implies \eqref{eqn:rate_convgc}.
\end{proof}

Given $\pi\in\mathfrak{F}_{m}$, define $N_{\pi}(\sigma,T)=\#\{\rho=\beta+i\gamma\colon \beta\geq\sigma,~|\gamma|\leq T,~L(\rho,\pi)=0\}$, where the zeros are counted with multiplicity.  For $L(s,\pi)$, GRH is equivalent to the assertion that $N_{\pi}(\sigma,T)=0$ for all $\sigma>\frac{1}{2}$.  The starting point for the proof of \cref{thm:weak} is the following inequality:  If $0\leq\alpha<\frac{1}{2}$, then
\begin{equation}
\label{eqn:ST_1/2}
\log|L(\tfrac{1}{2},\pi)|\leq\Big(\frac{1}{4}-\frac{\alpha}{10^{9}}\Big)\log C(\pi) + \frac{\alpha}{10^7} N_{\pi}(1-\alpha,6)+2\log|L(\tfrac{3}{2},\pi)|+O(m^2).
\end{equation}
This was proved by Soundararajan and the author \cite[Theorem 1.1]{ST}.  The combined work of Luo--Rudnick--Sarnak \cite{LRS} and M{\"u}ller--Speh \cite{MS} shows that $2\log|L(\tfrac{3}{2},\pi)|\ll m^2$.  Therefore, \eqref{eqn:ST_1/2} relates bounds for $L(\tfrac{1}{2},\pi)$ to the distribution of zeros of $L(s,\pi)$ near $s=1$.  Soundararajan and the author \cite[Corollary 2.6]{ST} also proved the {\it log-free zero density estimate}
\begin{equation}
\label{eqn:ST_LFZDE}
N_{\pi}(\sigma,T)\ll_{m} (C(\pi)T)^{10^7 m^3(1-\sigma)}.
\end{equation}
They showed that \eqref{eqn:ST_1/2} and \eqref{eqn:ST_LFZDE} imply \eqref{eqn:ST_weak} (choose $\alpha=(10^8 m^3\log C(\pi))^{-1}\log \log C(\pi)$).

Variations of \eqref{eqn:ST_LFZDE} in which one averages over $\pi\in\mathfrak{F}_m(Q)$ have been known for a long time assuming unproven progress toward the generalized Ramanujan conjecture (GRC) \cite{BTZ,KM}.  The author and Zaman \cite[Theorem 1.2] {TZ_GLn} recently proved such an estimate unconditionally:
\begin{equation}
\label{eqn:LFZDE}
\sum_{\pi\in\mathfrak{F}_{m}(Q)}N_{\pi}(\sigma,T)\ll_m (QT)^{10^7 m^4(1-\sigma)}.
\end{equation}
\cref{thm:weak} follows from the following refinement of \eqref{eqn:LFZDE} arising from Landau--Siegel zeros.
\begin{theorem}
	\label{thm:LFZDE}
Let $Q,T\geq 3$ and $0\leq\sigma\leq 1$.  Let $m\geq 2$.  If $\chi\pmod{q}$ is a primitive quadratic Dirichlet character having a real zero $\beta_{\chi}\in(\frac{1}{2},1)$ and $q\leq Q$, then
\[
\sum_{\pi\in\mathfrak{F}_{m}(Q)}N_{\pi}(\sigma,T)\ll_m \min\{1,(1-\beta_{\chi})\log(qQT)\}(qQT)^{10^7 m^4(1-\sigma)}.
\]
\end{theorem}
Bombieri \cite[Theorem 14]{Bombieri2} proved a version of \cref{thm:LFZDE} for zeros of Dirichlet $L$-functions with $\beta_{\chi}$ omitted from the count (if it exists).  A classical result of Deuring and Heilbronn quantifies the extent to which $\beta_{\chi}$ repels zeros of other Dirichlet $L$-functions from $\mathrm{Re}(s)=1$ \cite[Ch. 18]{IK}.  Bombieri's result recovers this phenomenon as a corollary, along with Siegel's ineffective upper bound on $\beta_{\chi}$.  When $m\leq 4$, \cref{thm:LFZDE} recovers recent work of Brumley, the author, and Zaman \cite[Theorem 1.7 when $\pi_0$ is trivial]{BTZ}.  When $m\geq 5$, the proofs in \cite{BTZ} relied on unproven progress toward GRC, which we remove here.  This opens up opportunities to study the arithmetic consequences of Landau--Siegel zeros (should they exist) in many other settings.  \cref{thm:QUE} ultimately uses \cref{thm:LFZDE} with $m=3$ and $m=6$.

Unlike Bombieri's setting, we can ensure that $L(\beta_{\chi},\pi)\neq 0$ for all $\pi\in\mathfrak{F}_m(Q)$, provided that $m\geq 2$ and $m^2(1-\beta_{\chi})\log(qQ)<c$ for a certain absolute and effectively computable constant $c>0$ (see \cref{lem:PAGE} below). Combined with \cref{thm:LFZDE}, this fact enables us to deduce for each $\pi\in\mathfrak{F}_m(Q)$ a much stronger zero-free region for $L(s,\pi)$ than unconditional methods permit, akin to the aforementioned result of Deuring and Heilbronn, when $\beta_{\chi}$ is sufficiently close to 1.

\subsection*{Acknowledgements}

I thank Kevin Ford, Peter Humphries, and the anonymous referee for their helpful comments.

\section{Properties of $L$-functions}
\label{sec:L-functions}


We recall some standard facts about $L$-functions arising from automorphic representations and their twists by Dirichlet characters; see \cite{ST} for a convenient summary.

\subsection{Standard $L$-functions}
\label{subsec:standard}
Let $\pi\in\mathfrak{F}_{m}$ have arithmetic conductor $q_{\pi}$ with $m\geq 2$.  We express $\pi$ is a tensor product $\otimes_{v}\pi_{v}$ of smooth admissible representations of $\mathrm{GL}_m(\mathbb{Q}_{p})$ and $\mathrm{GL}_m(\mathbb{R})$, where $v$ varies over the places of $\mathbb{Q}$.  If $v$ is nonarchimedean and corresponds with a prime $p$, in which case we write $\pi_p$ instead of $\pi_v$, then there are $m$ Satake parameters $\alpha_{1,\pi}( p ),\ldots,\alpha_{m,\pi}( p )$ from which we define
\[
	L(s,\pi_{ p })=\prod_{j=1}^{m}\Big(1-\frac{\alpha_{j,\pi}( p )}{  p^{s}}\Big)^{-1}=\sum_{j=0}^{\infty}\frac{\lambda_{\pi}( p^j)}{  p^{js}}.
\]
We have $\alpha_{j,\pi}( p )\neq0$ for all $j$ whenever $ p \nmid q_{\pi}$, and it might be the case that $\alpha_{j,\pi}( p )=0$ for some $j$ when $ p | q_{\pi}$.  The finite part of the standard $L$-function $L(s,\pi)$ associated to $\pi$ is of the form
\[
L(s,\pi)=\prod_{ p } L(s,\pi_{ p })=\sum_{n=1}^{\infty}\frac{\lambda_{\pi}(n)}{ n^s},\qquad \mathrm{Re}(s)>1.
\]
If $v$ is archimedean, in which case we write $\pi_{\infty}$ instead of $\pi_v$, there are $m$ Langlands parameters $\mu_{\pi}(j)\in\mathbb{C}$ for $1\leq j\leq m$ from which we define
\[
L(s,\pi_{\infty}) = \pi^{-\frac{ms}{2}}\prod_{j=1}^{m}\Gamma\Big(\frac{s+\mu_{\pi}(j)}{2}\Big).
\]
We define for $t\in\mathbb{R}$ the analytic conductor
\begin{equation}
\label{eqn:analytic_conductor_def}
C(\pi,t)=q_{\pi}\prod_{j=1}^{m}(3+|it+\mu_{\pi}(j)|),\qquad C(\pi)=C(\pi,0).
\end{equation}
There exists $\theta_m\in[0,\frac{1}{2}-\frac{1}{m^2+1}]$ such that we have the uniform bounds
\begin{equation}
	\label{eqn:GRC}
|\alpha_{j,\pi}( p )|\leq p^{\theta_m},\qquad \mathrm{Re}(\mu_{\pi}(j))\geq -\theta_m.
\end{equation}
(This follows from \cite{LRS} when $v$ is unramified and \cite{MS} when $v$ is ramified.)  GRC asserts that one may take $\theta_m=0$.

The completed $L$-function $\Lambda(s,\pi) = q_{\pi}^{s/2}L(s,\pi)L(s,\pi_{\infty})$ is entire of order 1.  Each pole of $L(s,\pi_{\infty})$ is a trivial zero of $L(s,\pi)$.  Since $\Lambda(s,\pi)$ is entire of order 1, a Hadamard factorization
\[
\Lambda(s,\pi)=e^{a_{\pi}+b_{\pi}s}\prod_{\rho}\Big(1-\frac{s}{\rho}\Big)e^{s/\rho}
\]
exists, where $\rho$ varies over the nontrivial zeros of $L(s,\pi)$.  These zeros satisfy $0\leq \mathrm{Re}(\rho)\leq 1$. 

\subsection{Twisted $L$-functions}
\label{subsec:RS}

If $\pi\in\mathfrak{F}_{m}$ and $\chi\pmod{q}$ is a primitive Dirichlet character, then $\pi\otimes\chi\in\mathfrak{F}_{m}$.  If $ p \nmid q_{\pi}q$, then we have the equality of sets $\{\alpha_{j,\pi\otimes\chi}( p )\}=\{\alpha_{j,\pi}( p )\chi( p )\}$.  See \cite[Appendix]{ST} for a description of $\alpha_{j,\pi\otimes\chi}( p )$ when $ p |q_{\pi}q$.   Note that if $\chi_0$ is the trivial character modulo 1, then $L(s,\pi\otimes\chi_0)=L(s,\pi)$.  The Langlands parameters $\mu_{\pi\otimes\chi}(j)$ are  given by $\{\mu_{\pi\otimes\chi}(j)\} = \{\mu_{\pi}(j)+\frac{1-\chi(-1)}{2}\}$ whenever $\pi_{\infty}$ is unramified.  As with $L(s,\pi)$, we define
\[
C(\pi\otimes\chi,t)=q_{\pi\otimes\chi}\prod_{j=1}^{m} (3+|it+\mu_{\pi\otimes\chi}(j)|),\quad C(\pi\otimes\chi)=C(\pi\otimes\chi,0).
\]
Bushnell and Henniart \cite[Theorem 1]{BH} proved that $q_{\pi\otimes\chi}\leq q^{m}q_{\pi}$, which implies that
\begin{equation}
\label{eqn:BH}
C(\pi\otimes\chi,t)\leq C(\pi\otimes\chi)(3+|t|)^{m},\qquad C(\pi\otimes\chi)\leq e^{O(m)} q^m C(\pi).
\end{equation}

Define the functions $\Lambda_{\pi}(n)$ and $\Lambda_{\pi\otimes\chi}(n)$ by the Dirichlet series identities
\[
\sum_{n=1}^{\infty}\frac{\Lambda_{\pi}(n)}{n^s}=-\frac{L'}{L}(s,\pi),\qquad \sum_{n=1}^{\infty}\frac{\Lambda_{\pi\otimes\chi}(n)}{n^s}=-\frac{L'}{L}(s,\pi\otimes\chi),\qquad \mathrm{Re}(s)>1.
\]

\begin{lemma}
	\label{lem:mertens}
	If $\pi\in\mathfrak{F}_{m}$, $\chi$ is a primitive quadratic character, and $\eta>0$, then
	\[
	\sum_{n=1}^{\infty}\frac{|\Lambda_{\pi\otimes\chi}(n)|}{n^{1+\eta}}\leq\frac{1}{\eta}+\frac{m}{2}\log C(\pi)+O(m^2).
	\]
\end{lemma}
\begin{proof}
	This is an immediate corollary of \cite[Lemma 4.2]{BTZ}.  In the notation therein, take $\pi_0$ to be the trivial representation of $\mathrm{GL}_1(\mathbb{A}_{\mathbb{Q}})$, so that $m_0=1$ and $C(\pi_0)\asymp 1$.
\end{proof}

\begin{lemma}
\label{lem:linnik_lemma}
	If $t\in\mathbb{R}$ and $0<\eta\leq 1$, then
	\[
	\#\{\rho\colon |\rho-(1+it)|\leq\eta,~L(\rho,\pi\otimes\chi)=0\}\leq 5m\eta \log(q C(\pi))+O(m^2\eta+1)
	\]
	and
	\[
	\sum_{\rho}\frac{1+\eta-\beta}{|1+\eta+it-\rho|^2}\ll \frac{1}{\eta}+m\log(q C(\pi)(3+|t|)).
	\]
\end{lemma}
\begin{proof}
	By the proof of \cite[Lemma 3.1]{ST}, we have
	\begin{align*}
\sum_{\rho}\frac{1+\eta-\beta}{|1+\eta+it-\rho|^2}&\leq \frac{1}{5\eta}\#\{\rho\colon |\rho-(1+it)|\leq \eta,~L(\rho,\pi\otimes\chi)=0\}\\
	&\leq \frac{1}{2}\log q_{\pi\otimes\chi}+\frac{1}{2}\sum_{j=1}^{m}\mathrm{Re}\frac{\Gamma'}{\Gamma}\Big(\frac{1+\eta+it+\mu_{\pi\otimes\chi}(j)}{2}\Big)+\sum_{n=1}^{\infty}\frac{|\Lambda_{\pi\otimes\chi}(n)|}{n^{1+\eta}}
	\end{align*}
The result follows from the bound $q_{\pi\otimes\chi}\leq q^m q_{\pi}$, Stirling's formula, and \cref{lem:mertens}.
\end{proof}

\begin{lemma}
\label{lem:PAGE}
Let $\chi\pmod{q}$ be a primitive quadratic character such that $L(s,\chi)$ has a zero $\beta_{\chi}\in(\frac{1}{2},1)$, and let $m\geq 2$.  There exists an absolute and effectively computable constant $c>0$ such that if $m^2(1-\beta_{\chi})\log(qQ)< c$, then $L(\beta_{\chi},\pi)\neq 0$ for all $\pi\in\mathfrak{F}_m(Q)$.
\end{lemma}
\begin{proof}
	It follows from \cite[Theorem A]{Hoffstein} that for at most one $\pi\in\mathfrak{F}_m(Q)\cup\{\chi\}$, $L(s,\pi)$ has a zero in the interval $[1-c(m^2\log qQ)^{-1},1)$.  If $m^2(1-\beta_{\chi})\log(qQ)< c$, then $\beta_{\chi}$ lies in this interval.  Since $L(\beta_{\chi},\chi)=0$ by definition and $\chi\notin\mathfrak{F}_m(Q)$ once $m\geq 2$, the lemma follows.
\end{proof}

\section{Proof of \cref{thm:LFZDE}}
\label{sec:LFZDE}

Let $\pi\in\mathfrak{F}_{m}(Q)$ with $m\geq 2$, and let $\chi\pmod{q}$ be a primitive quadratic Dirichlet character.  Suppose that $L(s,\chi)$ has a real zero $\beta_{\chi}$, and suppose that $L(s,\pi)$ has a zero $\rho_0$ (necessarily distinct from $\beta_{\chi}$ by \cref{lem:PAGE}) such that $|\rho_0-(1+i\tau)|\leq\eta$, where $\tau\in\mathbb{R}$, $|\tau|\leq T$, and
\begin{equation}
\label{eqn:eta_constraints}
(\log q Q T)^{-1}\leq \eta\leq(200m^2)^{-1}.
\end{equation}

\subsection{A zero detection criterion}

Let $k\geq 0$ be an integer, let $s=1+\eta+i\tau$, and define
\[
F(z)=L(z,\pi)L(z+1-\beta_{\chi},\pi\otimes\chi),\qquad G_k(z)=\frac{(-1)^k}{k!}\Big(\frac{F'}{F}\Big)^{(k)}(z).
\]
A standard calculation involving the Hadamard product as in \cite[Section 4]{ST} leading up to Equation 4.2\footnote{Use \cref{lem:linnik_lemma} in place of \cite[(3.5)]{ST}.} shows that if $k\geq 0$ is an integer, then
\begin{equation}
\label{eqn:explicit_formula}
G_k(s) = \sum_{\substack{L(\rho,\pi)=0 \\ |s-\rho|\leq 200\eta}}\frac{1}{(s-\rho)^{k+1}}+\sum_{\substack{L(\rho',\pi\otimes\chi)=0 \\ |s-\rho'|\leq 200\eta}}\frac{1}{(s+1-\beta_{\chi}-\rho')^{k+1}}+O\Big(\frac{m^2\eta\log(q QT)}{(200\eta)^k}\Big).
\end{equation}
The hypothesized existence of $\rho_0$ ensures that $|G_k(s)|$ has a large lower bound.

\begin{lemma}
\label{lem:lower_bound}
	Let $\eta$ satisfy \eqref{eqn:eta_constraints}, and let $\tau\in\mathbb{R}$ satisfy $|\tau|\leq T$.  Suppose that $L(z,\pi)$ has a zero $\rho_0$ satisfying $|\rho_0-(1+i\tau)|\leq\eta$.  If $K>\lceil 2000m^2\eta\log(qQT)+O_m(1)\rceil$ with a sufficiently large implied constant, then $\eta^{k+1}|G_k(1+\eta+i\tau)|\geq (2(100)^{k+1})^{-1}$ for some $k\in[K,2K]$.
\end{lemma}
\begin{proof}
By \cref{lem:linnik_lemma}, there are at most $2000m^2\eta\log(qQT)+O(1)$ zeros of $F(z)$ such that $|1+\eta+i\tau-\rho|\leq 200\eta$.  A result of S{\'o}s and Tur{\'a}n \cite[Theorem]{Turan} states that if $z_1,\ldots,z_{\nu}\in\mathbb{C}$ and $K\geq \nu$, then there exists an integer $k\in[K,2K]$ such that $|z_1^k+\cdots+z_{\nu}^k|\geq (|z_1|/50)^k$.  We apply this to \eqref{eqn:explicit_formula} with $K\geq\lceil2000m^2\eta\log(qQT)+O_m(1)\rceil$ and $z_1 = 1/(s-\rho_0)$, which has modulus at least $1/(2\eta)$.  (Recall that $s=1+\eta+i\tau$.)  It follows that
\[
\Big|\sum_{\substack{L(\rho,\pi)=0 \\ |s-\rho|\leq 200\eta}}\frac{1}{(s-\rho)^{k+1}}+\sum_{\substack{L(\rho',\pi\otimes\chi)=0 \\ |s-\rho'|\leq 200\eta}}\frac{1}{(s+1-\beta_{\chi}-\rho)^{k+1}}\Big|\geq\Big(\frac{1}{50|s-\rho_0|}\Big)^{k+1}\geq\frac{1}{(100\eta)^{k+1}}.
\]
We apply this to \eqref{eqn:explicit_formula}, and the lemma follows.
\end{proof}

We now prove an upper bound for $|G_k(s)|$.

\begin{lemma}
	\label{lem:upper_bound}
	Let $\eta$ satisfy \eqref{eqn:eta_constraints}, and let $\tau\in\mathbb{R}$ satisfy $|\tau|\leq T$.  Let $K$ be as in \cref{lem:lower_bound}, and define $N_0=\exp(K/(300\eta))$ and $N_1=\exp(40K/\eta)$.  If $k\in[K,2K]$, then
	\[
	\eta^{k+1}|G_k(1+\eta+i\tau)|\leq \eta^2\int_{N_0}^{N_1}\Big|\sum_{N_0\leq p\leq u}\frac{\lambda_{\pi}(p)\log p}{p^{1+i\tau}}(1+\chi(p)p^{\beta_{\chi}-1})\Big|\frac{du}{u}+O\Big(\frac{k}{(110)^{k}}\Big).
	\]
\end{lemma}
\begin{proof}
	A direct computation shows that
	\[
	\eta^{k+1}|G_k(s)|=\eta\Big|\sum_{n=1}^{\infty}\frac{\Lambda_{\pi}(n)+\Lambda_{\pi\otimes\chi}(n)n^{\beta_{\chi}-1}}{n^{1+\eta+i\tau}}\frac{(\eta\log n)^k}{k!}\Big|.
	\]
	Since $|\chi(n)|\leq 1$ and $\beta_{\chi}<1$, it follows from \cite[(5.6) and (5.7)]{TZ_GLn} that
	\[
	|\eta^{k+1}G_k(s)|=\eta\Big|\sum_{p\in[N_0,N_1]}\frac{\Lambda_{\pi}(p)+\Lambda_{\pi\otimes\chi}(p)p^{\beta_{\chi}-1}}{p^{1+\eta+i\tau}}\frac{(\eta\log p)^k}{k!}\Big|+O\Big(\frac{m^2\eta\log(qQT)}{(110)^k}\Big).
	\]
	Since $N_0>\max\{Q,q\}$, we have $\Lambda_{\pi}(p)+\Lambda_{\pi\otimes\chi}(p)p^{\beta_{\chi}-1}=\lambda_{\pi}(p)(1+\chi(p)p^{\beta_{\chi}-1})\log p$ when $p\in[N_0,N_1]$.  The lemma now follows by partial summation and \cref{lem:mertens}, just as in \cite[Lemma 5.5]{TZ_GLn}.
\end{proof}

\cref{lem:lower_bound,lem:upper_bound} provide a criterion for detecting zeros.  Let $K = 10^5 \eta m^4 \log (qQT)+O_m(1)$ with a sufficiently large implied constant.  If $k\in[K,2K]$, then $O(k(110)^{-k})\leq \frac{1}{4}(100)^{-k-1}$.  Therefore, if $k\in [K,2K]$, $|\tau|\leq T$, $\eta$ satisfies \eqref{eqn:eta_constraints}, and $L(s,\rho)$ has a zero $\rho_0$ satisfying $|\rho_0-(1+i\tau)|\leq\eta$, then it follows from \cref{lem:lower_bound,lem:upper_bound} that
\[
1\leq 4(100)^{2K+1}\eta^2\int_{N_0}^{N_1}\Big|\sum_{N_0\leq p\leq u}\frac{\lambda_{\pi}(p)\log p}{p^{1+i\tau}}(1+\chi(p)p^{\beta_{\chi}-1})\Big|\frac{du}{u}.
\]
We square both sides, apply Cauchy--Schwarz, and use \cref{lem:linnik_lemma} to deduce that
\begin{multline}
\label{eqn:criterion}
\frac{\#\{\rho=\beta+i\gamma\colon \beta\geq 1-\frac{\eta}{2},~|\gamma-\tau|\leq \frac{\eta}{2}\}}{m^2\eta\log(qQT)}\\
\ll (101)^{4K}\eta^3\int_{N_0}^{N_1}\Big|\sum_{N_0\leq  p \leq u}\frac{\lambda_{\pi}( p )\log p }{  p^{1+i\tau}}(1+\chi(p)p^{\beta_{\chi}-1})\Big|^2\frac{du}{u}.
\end{multline}

\subsection{Proof of \cref{thm:LFZDE}}
In order to count detected zeros in the box $[1-\frac{\eta}{2},1]\times[-T,T]$, we integrate \eqref{eqn:criterion} over $|\tau|\leq T$ and observe that $m^2\eta \log(qQT)\ll K$ to obtain
	\begin{equation}
	\label{eqn:5.6}
N_{\pi}(1-\tfrac{\eta}{2},T)\ll (101)^{4K}K\eta^2 \int_{N_0}^{N_1}\int_{-T}^{T}\Big|\sum_{N_0\leq  p \leq u}\frac{\lambda_{\pi}( p )\log p }{  p^{1+i\tau}}(1+\chi(p)p^{\beta_{\chi}-1})\Big|^2 d\tau\frac{du}{u}.
\end{equation}
We sum \eqref{eqn:5.6} over $\pi\in\mathfrak{F}_{m}(Q)$ to find that $\sum_{\pi\in\mathfrak{F}_{m}(Q)}N_{\pi}(1-\tfrac{\eta}{2},T)$ is
	\begin{equation}
		\label{eqn:estimate_0}
\ll (101)^{4K}K\eta^2 \int_{N_0}^{N_1}\Big(\sum_{\pi\in\mathfrak{F}_{m}(Q)}\int_{-T}^{T}\Big|\sum_{N_0\leq  p \leq u}\frac{\lambda_{\pi}( p )\log p }{p^{1+i\tau}}(1+\chi(p)p^{\beta_{\chi}-1})\Big|^2 d\tau\Big) \frac{du}{u}.
	\end{equation}

\begin{lemma}
\label{lem:MVT}
If $u\in[N_0,N_1]$, then with the notation above, we have
\[
\sum_{\pi\in\mathfrak{F}_{m}(Q)}\int_{-T}^{T}\Big|\sum_{N_0\leq  p \leq u}\frac{\lambda_{\pi}( p )\log p }{p^{1+i\tau}}(1+\chi(p)p^{\beta_{\chi}-1})\Big|^2 d\tau\ll_{m}\frac{\eta}{K}\sum_{p\in[N_0,u]}\frac{(\log p)^2}{p}(1+\chi(p)p^{\beta_{\chi}-1})^{2}.
\]
\end{lemma}
\begin{proof}
By \cite[Theorem 1]{Gallagher}, if $a\colon\mathbb{Z}\to\mathbb{C}$ is an arbitrary function and $\sum_{p}|a(p)|<\infty$, then
	\[
	\sum_{\pi\in\mathfrak{F}_m(Q)}\int_{-T}^{T}\Big|\sum_{p>z}\lambda_{\pi}(p)a(p)p^{-i\tau}\Big|^2 d\tau\ll T^2\int_0^{\infty}\sum_{\pi\in\mathfrak{F}_m(Q)}\Big|\sum_{\substack{p\in(x,xe^{1/T}] \\ p>z}}\lambda_{\pi}(p)a(p)\Big|^2\frac{dx}{x}.
	\]
	If $Q,T,x\geq 1$ and $z\gg_{m,\varepsilon}Q^{2(m^2+m+\varepsilon)}$ with a large implied constant, then by \cite[(4.7)]{TZ_GLn},
	\[
	\sum_{\pi\in\mathfrak{F}_{m}(Q)}\Big|\sum_{\substack{p\in(x,xe^{1/T}] \\ p>z}}\lambda_{\pi}(p)a(p)\Big|^2 \ll_{m,\varepsilon}\Big(\frac{x}{T\log z}+Q^{m^2+m+\varepsilon}T^{m^2}z^{2m^2+2+\varepsilon}|\mathfrak{F}_m(Q)|\Big)\sum_{\substack{p\in(x,xe^{1/T}] }}|a(p)|^2.
	\]
	Combining these bounds with $|\mathfrak{F}_m(Q)|\ll_{m,\varepsilon}Q^{2m+\varepsilon}$ \cite[Theorem A.1]{BTZ} and $\varepsilon=\frac{1}{100}$, we find that
	\begin{align*}
	\sum_{\pi\in\mathfrak{F}_m(Q)}\int_{-T}^{T}\Big|\sum_{p>z}\lambda_{\pi}(p)a(p)p^{-i\tau}\Big|^2 d\tau &\ll_{m} \sum_{p}|a(p)|^2 p\Big(\frac{1}{\log z}+\frac{Q^{5m^2}T^{m^2}z^{2m^2+3}}{p}\Big)\\
	&\ll_{m} \frac{1}{\log z}\sum_{p}|a(p)|^2 p\Big(1+\frac{Q^{5m^2}T^{m^2}z^{2m^2+4}}{p}\Big).
	\end{align*}
	Choose $y=N_0$ and $z=y^{1/(10m^2)}$.  Given $u\in[N_0,N_1]$, it remains to choose $a(p)=0$ unless $p\in[y,u]$, in which case $a(p)=\frac{\log p}{p}(1+\chi(p)p^{\beta_{\chi}-1})$.
\end{proof}

Recall the definitions of $K$, $\eta$, $N_0$, and $N_1$.  We apply \cref{lem:MVT} to \eqref{eqn:estimate_0} and find that
\[
\sum_{\pi\in\mathfrak{F}_m(Q)}N_{\pi}(1-\tfrac{\eta}{2},T)\ll (101)^{4K}\eta^3\int_{N_0}^{N_1}\sum_{p\in[N_0,u]}\frac{(\log p)^2}{p}(1+\chi(p)p^{\beta_{\chi}-1})^2\frac{du}{u}.
\]
It follows from a lemma of Bombieri \cite[Lemme C, p. 50]{Bombieri2} that
\[
\sum_{p\in[N_0,N_1]}\frac{(\log p)^2}{p}(1+\chi(p)p^{\beta_{\chi}-1})^2\ll_m \frac{K^3}{\eta^2}\min\Big\{1,\frac{1-\beta_{\chi}}{\eta}\Big\},
\]
so
\begin{align*}
\sum_{\pi\in\mathfrak{F}_m(Q)}N_{\pi}(1-\tfrac{\eta}{2},T)\ll(101)^{4K}K^4\min\{1,(1-\beta_{\chi})\log(qQT)\}.
\end{align*}
Using our choices of $K$ and $\eta$ and writing $\sigma=1-\frac{\eta}{2}$, we conclude that
\[
\sum_{\pi\in\mathfrak{F}_{m}(Q)}N_{\pi}(\sigma,T)\ll_{m} \min\{1,(1-\beta_{\chi})\log(qQT)\}(qQT )^{10^7 m^4(1-\sigma)}
\]
when $1-(400m^2)^{-1}\leq \sigma<1-(2\log qQT)^{-1}$.  If $\sigma\geq 1-(2\log qQT)^{-1}$, then $N_{\pi}(\sigma,T)\leq N_{\pi}(1-(2\log qQT)^{-1},T)$.   If $\sigma<1-(400n^2)^{-1}$, then the result is trivial by the Riemann--von Mangoldt formula \cite[Theorem 5.8]{IK} and the effective bound $1-\beta_{\chi}\gg q^{-1/2}$ \cite{Davenport}.

\section{Proof of \cref{thm:weak}}
\label{sec:weak}

We begin with a bound relating the size of $L$-functions on the line $\mathrm{Re}(s)=\frac{1}{2}$ to 

\begin{lemma}
\label{lem:ST1/2}
	Let $m\geq 2$ be an integer, $t\in\mathbb{R}$, and $\pi\in\mathfrak{F}_m$.  If $0\leq\alpha<\frac{1}{2}$, then
	\[
	\log|L(\tfrac{1}{2}+it,\pi)|\leq \Big(\frac{1}{4}-\frac{\alpha}{10^9}\Big)\log(C(\pi)(3+|t|)^m)+\frac{\alpha}{10^7}N_{\pi}(1-\alpha,|t|+6)+O(m^2).
	\]
\end{lemma}
\begin{proof}
Recall \eqref{eqn:ST_1/2}.  If $t\in\mathbb{R}$, the bound applies to $L(\frac{1}{2}+it,\pi)$ as well once we add $it$ to all of the Langlands parameters $\mu_{\pi}(j)$ for $1\leq j\leq m$.  Consequently, we have
\begin{equation*}
\begin{aligned}
\log|L(\tfrac{1}{2}+it,\pi)|&\leq\Big(\frac{1}{4}-\frac{\alpha}{10^{9}}\Big)\log C(\pi,t) +2\log|L(\tfrac{3}{2}+it,\pi)|+O(m^2)\\
&+ \frac{\alpha}{10^7} \#\{\rho=\beta+i\gamma\colon L(\rho,\pi)=0,~\beta\geq1-\alpha,~|\gamma-t|\leq 6\}.
\end{aligned}
\end{equation*}
By \eqref{eqn:GRC}, we have $2\log|L(\tfrac{3}{2}+it,\pi)|\ll m^2$.  Furthermore, we have the trivial bound
\[
\#\{\rho=\beta+i\gamma\colon L(\rho,\pi)=0,~\beta\geq1-\alpha,~|\gamma-t|\leq 6\}\ll N_{\pi}(1-\alpha,|t|+6).
\]
The lemma now follows from the bound $C(\pi,|t|+6)\leq C(\pi)(|t|+6)^{m}$ in \eqref{eqn:BH}.
\end{proof}

Let $t\in\mathbb{R}$.  By \eqref{eqn:LFZDE}, all except $O_m((\log Q)^{10^{16}m^4\delta})$ of the $\pi\in\mathfrak{F}_m(Q)$ satisfy
\[
L(s,\pi)\neq0\qquad\textup{when}\quad|\mathrm{Im}(s)|\leq |t|+6,~\mathrm{Re}(s)\geq 1-\frac{10^{9}\delta\log\log C(\pi)}{\log(C(\pi)(|t|+6)^m)}.
\]
This and \cref{lem:ST1/2} prove \cref{thm:weak}(2).

For \cref{thm:weak}(1), let $\chi\pmod{q}$ be a primitive quadratic Dirichlet character whose $L$-function $L(s,\chi)$ has a zero $\beta_{\chi}\in(\frac{1}{2},1)$, and define $\lambda=(1-\beta_{\chi})\log q$.  If $Q=C(\pi)$, $C(\pi)\in[q,q^A]$ for some $A\geq 1$, and $(1-\beta_{\chi})\log(qC(\pi)T)$ is sufficiently small with respect to $m$, then \cref{thm:LFZDE} implies that there exists a constant $c_m>0$, depending at most on $m$, such that
\[
N_{\pi}(\sigma,T)\leq c_m \frac{\lambda}{\log q}\log(C(\pi)T)(C(\pi)T)^{2\times 10^7 m^4(1-\sigma)}.
\]
It follows that $\alpha N_{\pi}(1-\alpha,|t|+6)\ll_m 1$ when
\[
\alpha = \frac{\log(\frac{e\log q}{100\lambda c_m})}{4\times 10^7 m^4\log(C(\pi)(|t|+6)^m)}.
\]
With this choice of $\alpha$ in \cref{lem:ST1/2}, we find that
\begin{align*}
\log|L(\tfrac{1}{2}+it,\pi)|&\leq \frac{1}{4}\log(C(\pi)(|t|+3)^m) - \frac{1}{4\times 10^{16}m^4}\log\Big(\frac{e\log q}{100\lambda c_m }\Big)+O_m(1)\\
&\leq \frac{1}{4}\log(C(\pi)(|t|+3)^m) + \frac{1}{4\times 10^{16}m^4}\log\Big(\frac{A\lambda}{\log C(\pi)}\Big)+O_m(1).
\end{align*}
The claimed result follows once we exponentiate.
\bibliographystyle{abbrv}
\bibliography{JAThorner_v3_QUE_Siegel_zero}

\end{document}